\documentclass{tran-l}

\usepackage{amsmath,amsthm,amssymb,mathtools,tocvsec2,pdflscape,cite}
\usepackage{latexsym}
\usepackage{color,soul}
\usepackage[utf8]{inputenc}
\usepackage[T1]{fontenc}

\usepackage[breaklinks=true]{hyperref}
\allowdisplaybreaks

\DeclareMathOperator{\Id}{Id}
\theoremstyle{definition}

\newtheorem{defn}[equation]{Definition}

\newtheorem{remark}[equation]{Remark}
\theoremstyle{plain}
\newtheorem{lemma}[equation]{Lemma}
\newtheorem{theorem}[equation]{Theorem}
\newtheorem{utheorem}{\textrm{\textbf{Theorem}}}

\newtheorem{ucor}[utheorem]{\textrm{\textbf{Corollary}}}

\newtheorem{question}[equation]{Question}

\newcommand{\R}{\mathbb{R}}
\newcommand{\Z}{\mathbb{Z}}

\newcommand{\C}{\mathbb{C}}
\newcommand{\F}{\mathbb{F}}
\newcommand\Pn{\mathbb{P}}

\title{}
\numberwithin{equation}{section}

\begin{document}

\title[Preservers forbidden from diagonal blocks]{Positivity preservers\\forbidden to operate on diagonal blocks}
\author{Prateek Kumar Vishwakarma}
\address{Department of Mathematics, Indian Institute of Science, Bangalore, India.}
\email{prateekv@iisc.ac.in}


\begin{abstract}
The question of which functions acting entrywise preserve positive semidefiniteness has a long history, beginning with the Schur product theorem [\textit{Crelle} 1911], which implies that absolutely monotonic functions (i.e., power series with nonnegative coefficients) preserve positivity on matrices of all dimensions. A famous result of Schoenberg and of Rudin [\textit{Duke Math. J.} 1942, 1959] shows the converse: there are no other such functions.

Motivated by modern applications, Guillot and Rajaratnam [\textit{Trans. Amer. Math. Soc.} 2015] classified the entrywise positivity preservers in all dimensions, which act only on the off-diagonal entries. These two results are at ``opposite ends'', and in both cases the preservers have to be absolutely monotonic.

We complete the classification of positivity preservers that act entrywise except on specified ``diagonal/principal blocks'', in every case other than the two above. (In fact we achieve this in a more general framework.) This yields the first examples of dimension-free entrywise positivity preservers - with certain forbidden principal blocks - that are not absolutely monotonic.
\end{abstract}

\maketitle
\date{\today}

\section{Introduction}\label{Introduction}
Functions operating on a class of matrices which preserve a specified property of the class have been studied extensively in the literature. In this paper we discuss entrywise positivity preservers: functions that operate entrywise on matrices and preserve the class of positive semidefinite matrices in all dimensions. The search for such preservers began with the discovery of a product theorem due to Schur:

\begin{theorem}[Schur \cite{Schur}]\label{Schur}
Suppose $n\geq 1$ is an integer, and $A=(a_{ij}), B=(b_{ij})\in\C^{n\times n}$ are positive semidefinite matrices. Then $A \circ B:=(a_{ij} b_{ij})$ is positive semidefinite.
\end{theorem}

In particular, if $A$ is  positive semidefinite then so is $A^{\circ k}:=A \circ A \circ \dots \circ A =(a_{ij}^k)$ for all integers $k\geq 0$ under the convention that $0^0:=1$. This implies that monomials $x^k$, $k \geq 1$ preserve positivity when applied entrywise to any positive semidefinite matrix. Combining this with the properties of positive semidefinite matrices, we have the following:

\begin{lemma}[P\'{o}lya--Szeg\"{o} \cite{Polya-Szego}]\label{Polya-Szego}
Suppose $f(x):=\sum_{k\geq 0} c_{k}x^{k}$ is a power series with nonnegative coefficients $c_k$, that converges over $I\subseteq \R$. Then $f[A]:=(f(a_{ij}))$ is positive semidefinite for all positive semidefinite $A=(a_{ij})\in I^{n\times n}$ and all $n\geq 1$.
\end{lemma}
The functions appearing in Lemma~\ref{Polya-Szego} (power series with nonnegative coefficients) form the class of \textit{absolutely monotonic} functions. Hence Lemma~\ref{Polya-Szego} shows that absolutely monotonic functions preserve positivity when applied entrywise to a positive semidefinite matrix of any dimension. This observation naturally led to the search of a function that is not absolutely monotonic, but still preserves positivity. Schoenberg \cite{Schoenberg} proved that there are no such continuous functions. This result was later strengthened by removing the continuity assumption.

\begin{theorem}[Schoenberg \cite{Schoenberg}, Rudin \cite{Rudin}, Christensen--Ressel \cite{Chressel}]\label{SchoenbergPlus}
Let $I=(-\rho,\rho)$ for $0<\rho\leq \infty$, and $f:I\to\R$ be a function. Then the following are equivalent:
\begin{itemize}
\item[1.] $f[A]:=(f(a_{ij}))$ is positive semidefinite for all positive semidefinite $A=(a_{ij}) \in I^{n\times n}$, for all $n \geq 1$.
\item[2.] $f(x)=\sum_{k \geq 0} c_k x^k$ for all $x\in I$, where $c_k\geq 0$ for all $k\geq 0$.
\end{itemize}
\end{theorem}

We should mention that Rudin \cite{Rudin}, in addition to strengthening the result in \cite{Schoenberg}, conjectured that the entrywise positivity preservers over the class of complex positive semidefinite matrices are similar to those discovered previously, and coincide with the functions of the form:

\begin{align}\label{complexseries}
\sum_{m,k\geq 0} c_{m,k}z^{m}\overline{z}^{k}, \qquad \mbox{where } c_{m,k}\geq 0 \mbox{ for all }  m,k\geq 0.
\end{align}
In this article, we call a function of the form $c \cdot z^m\overline{z}^k$ with $c\geq 0$ a \textit{Herz function}. Observe that the conjugation map $z \mapsto \overline{z}$ when applied entrywise to a Hermitian matrix results in its transpose, which has the same spectrum. Hence the entrywise conjugation operation preserves the class of (complex) positive semidefinite matrices. The Schur product theorem combined with this shows that the functions in \eqref{complexseries} belong to such class of entrywise positivity preservers. This is an observation, analogous to Lemma~\ref{Polya-Szego}, in the complex setting. The reverse inclusion (i.e. Rudin's conjecture) was proved by Herz \cite{Herz} for positive semidefinite matrices having entries in the open unit disc. This result was later reproved for all the discs centered at the origin of positive or infinity radius.

\begin{theorem}[Herz \cite{Herz}, FitzGerald--Micchelli--Pinkus \cite{FMP}]\label{HerzPlus}
Let $I=D(0,\rho)$ for $0<\rho\leq \infty$, and $f:I \to \C$ be a function. Then the following are equivalent:
\begin{itemize}
\item[1.]
$f[A]:=(f(a_{ij}))$ is positive semidefinite for all positive semidefinite $A=(a_{ij})\in I^{n\times n}$, for all $n\geq 1$.
\item[2.]
$f(z)=\sum_{m,k\geq 0} c_{m,k}z^{m}\overline{z}^k$ for all $z\in I$, where $c_{m,k}\geq 0$ for all $m,k\geq 0$. 
\end{itemize}
\end{theorem}

In a parallel direction, Vasudeva \cite{Vasudeva} proved that the functions preserving the class of doubly nonnegative matrices (i.e. positive semidefinite matrices with positive entries) are exactly the power series with nonnegative Maclaurin coefficients. This result was later generalized and we have the following:

\begin{theorem}[Vasudeva \cite{Vasudeva}, Guillot--Khare--Rajaratnam \cite{GKR-Trans}]\label{VasudevaPlus}
Let $I=(0,\rho)$ or $I=[0,\rho)$ for $0<\rho\leq \infty$, and $f:I\to \R$ be a function. Then the following are equivalent.
\begin{itemize}
\item[1.] $f[A]:=(f(a_{ij}))$ is positive semidefinite for all positive semidefinite $A=(a_{ij})\in I^{n\times n}$, for all $n \geq 1$.
\item[2.] $f(x)=\sum_{k \geq 0} c_k x^k$ for all $x\in I$, where $c_k\geq 0$ for all $k\geq 0$.
\end{itemize}
\end{theorem}

To summarize: the functions operating entrywise and preserving the class of all positive semidefinite matrices with entries in $I$ are exactly the class of functions in \eqref{complexseries} over $I$, for $I=D(0,\rho), (-\rho,\rho),(0,\rho)$ and $[0,\rho)$.

The study of such entrywise positivity preservers has attracted recent attention owing to its applicability in big data analysis (for instance see \cite{GKR-JMAA, GKR-JCTA, GKR-Trans, Guillot-Rajaratnam-Trans}). One such instance that motivated our research is the work of Guillot--Rajaratnam \cite{Guillot-Rajaratnam-Trans}. They revisited and extended Theorem~\ref{SchoenbergPlus} by classifying those entrywise positivity preservers that operated only on the off-diagonal entries; we denote this operation by $f_{*}[A]$ for a square matrix $A:=(a_{ij})$ and define it formally as 
\begin{align*}
(f_{*}[A])_{ij}:=
\begin{cases}
a_{ij} & \mbox{ if }i=j, \\
f(a_{ij}) & \mbox{ otherwise}.\\ 
\end{cases}
\end{align*}

Guillot--Rajaratnam found that these preservers, like in Theorem~\ref{SchoenbergPlus}, are necessarily absolutely monotonic. More precisely:

\begin{theorem}[Guillot--Rajaratnam \cite{Guillot-Rajaratnam-Trans}]\label{Guillot-Rajaratnam}
Let $I=(-\rho,\rho)$ for $0<\rho\leq \infty$, and $f:I\to \R$ be a function. Then the following are equivalent:
\begin{itemize}
\item[1.]
$f_{*}[A]$ is positive semidefinite for all positive semidefinite $A\in I^{n\times n}$, for all $n \geq 1$.
\item[2.]
$f(x)=\sum_{k\geq 1} c_{k}x^{k}$ and $|f(x)|\leq |x|$ for all $x\in I$, where $c_{k}\geq 0$ for all $k\geq 0$.
\end{itemize}
\end{theorem}

In this paper, we explain that the work of Schoenberg (and others) and Guillot--Rajaratnam are two ``extreme'' cases among other possibilities. We do so essentially by forbidding the entrywise functions $f$ from operating on ``diagonal/principal blocks''. For instance in Theorem~\ref{SchoenbergPlus}, $f$ is acting on all the entries in the matrices, which means it is not forbidden from any entry or, in particularly, from any diagonal/principal block. However in Theorem~\ref{Guillot-Rajaratnam}, $f$ is operating only on the off-diagonal entries, i.e. it is forbidden to operate on all $1\times 1$ diagonal blocks. Similarly, one can ask for the preservers $f$ which are forbidden in more generality: for example, forbidden to operate on some $1\times 1$ diagonal blocks but not all of them, or forbidden to operate on some $k\times k$ diagonal/principal blocks, etc. 

This idea makes it possible to unify the two different-looking results (Theorem~\ref{SchoenbergPlus} and Theorem~\ref{Guillot-Rajaratnam}) into one framework. In this process of unification, we provide dimension-free, non-absolutely monotonic positivity preservers when $f$ is forbidden from acting on certain diagonal/principal blocks. To our knowledge, this is the first time that non-absolutely monotonic functions are found to be preserving positivity in all dimensions, when operated in a certain way.

The remainder of this paper is organized as follows. The next section introduces the required notations and states our main results and some of the corresponding key features. Section~\ref{ProofofT1} proves the results when $f$ is forbidden only from $1\times 1$ diagonal blocks and the final two sections prove the results when $f$ is forbidden from larger diagonal/principal blocks.

\section{Main Results over Complex Disc Domains}\label{Mainresults}
In this section we state our main results over complex disc domains (to be proved below); see Table~\ref{Table1} for a summary of our results. Define $\Pn_n(I):=\Pn_n\cap I^{n\times n}$ for $I\subseteq \C$, where $\Pn_n$ denotes the class of all $n\times n$ (Hermitian) positive semidefinite matrices, and define $[n]:=\{1,2,\dots,n\}$, for all integers $n\geq 1$. Let $\Id_{n}$ denote the $n\times n$ identity matrix, and define constant functions $\mathbf{0}:z \mapsto 0$ and $\mathbf{1}:z \mapsto 1$ for all $z\in \C$. We also use $\mathbf{0}$ and $\mathbf{1}$ to denote the matrices with all entries zero and one respectively; the dimension should be clear from the context.

\begin{defn}\label{Def1}
Let $I \subseteq \C$ and $f:I \to \F$ be a function, where $\F=\C$ if $I \not \subseteq \R$, otherwise $\F=\R$. Suppose $T_n\subseteq 2^{[n]}$ for all integers $n\geq 1$. We define the entrywise operation,
\begin{align*}
f_{T_n}[-]:I^{n\times n}\to \C^{n\times n},
\end{align*}
where 
\begin{align*}
(f_{T_n}[A])_{ij}:=\begin{cases}
a_{ij} & \mbox{ if } i,j\in U \mbox{ for some }U\in T_n, \\
f(a_{ij}) & \mbox{ otherwise,}\\
\end{cases}
\end{align*}
$\mbox{ for all } A=(a_{ij}) \in I^{n\times n}$. We use $f_{*}[-]$ and $f[-]$ to denote $f_{T_n}[-]$ when $T_n=\{\{j\}:j\in [n]\}$ and $T_n=\emptyset$ respectively.
\end{defn}
Our goal in this paper is to resolve the following problem:
\begin{question}\label{Question}
Let $I\subseteq\C$, and $\mathbb{F}=\C$ if $I\not\subseteq\R$, otherwise $\mathbb{F}=\R$. Let $T_n\subseteq 2^{[n]}$ for all $n\geq 1$. Characterize the functions $f:I\to \mathbb{F}$ such that for any specified sequence $(T_n)_{n\geq 1}$, the action $f_{T_n}[-]$ preserves the positivity of the matrices in $\Pn_n(I)$, for all $n\geq 1$.
\end{question}

As discussed above, Theorems~\ref{SchoenbergPlus}, \ref{HerzPlus}, \ref{VasudevaPlus} and Theorem~\ref{Guillot-Rajaratnam} answer Question~\ref{Question} for the two ``extreme'' cases where $T_n$ is either empty for all $n$ or $\{\{j\}:j\in [n]\}$ for all $n$. We obtain a complete classification for any choice of $(T_n)_{n\geq 1}$ below. One of the curious features of this classification is that the above two ``extreme'' cases yield absolutely monotonic functions, which is not always so in the ``intermediate'' cases. That is, we obtain in this paper novel, non-absolutely monotonic families of dimension-free positivity preservers, for certain cases of $(T_n)_{n\geq 1}$.

Answering Question~\ref{Question} also brings the two extreme results (Theorem~\ref{SchoenbergPlus} and Theorem~\ref{Guillot-Rajaratnam}) together under one roof. In fact, this classification will follow from stronger results proved below. We state these results after introducing the required new notations. See Corollary~\ref{The-corollary} or the last two columns of Table~\ref{Table1} for the classifications that answer Question~\ref{Question}.

\begin{defn}\label{Def2}
Let $I \subseteq \C$ and $g,f:I \to \F$ be functions, where $\F=\C$ if $I \not\subseteq \R$, otherwise $\F=\R$. Suppose $T_n\subseteq 2^{[n]}$ for all integers $n\geq 1$. 
\begin{enumerate}
\item[1.]
We define the entrywise operation,
\begin{align*}
(g,f)_{T_n}[-]:I^{n\times n}\to {\C}^{n\times n},
\end{align*}
where 
\begin{align*}
((g,f)_{T_n}[A])_{ij}:=\begin{cases}
g(a_{ij}) & \mbox{ if } i,j\in U \mbox{ for some }U\in T_n, \\
f(a_{ij}) & \mbox{ otherwise,}\\
\end{cases}
\end{align*}
$\mbox{ for all } A=(a_{ij}) \in I^{n\times n}$.

\item[2.]
Let $(g,f)_{*}[-]$ denote $(g,f)_{T_n}[-]$ when $T_n=\{\{j\}:j\in [n]\}$.

\item[3.]
{(}Nonempty and empty sequences.{)}
If there exists an integer $N\geq 2$ such that $T_{N}\neq \emptyset $, then we say that the sequence $(T_n)_{n\geq 1}$ is a nonempty sequence, otherwise we call it empty.
\end{enumerate}
 
\end{defn}

We work in a general setting, and characterize the tuples $(g,f)$ such that the operations $(g,f)_{T_n}[-]$ preserve the positivity of the matrices in $\Pn_n(I)$, for all $n\geq 1$. Clearly, Question~\ref{Question} is the special case when $g\equiv\Id$. Before moving any further, we collect some remarks and natural assumptions on the sequence $(T_n)_{n\geq 1}$ and on $\Pn_n(I)$ for $I\subseteq\C$.

\begin{remark}\label{Sequencerem}
\item[{1.}]
The nature of Question~\ref{Question} allows to assume $U\not\subseteq V$ for all $U\neq V\in T_n$ for all $n \geq 1$, and that $T_n=\{\emptyset\}$ and $T_n=\emptyset$ are equal. We will assume this henceforth without further mention.
\item[{2.}] 
Since we are characterizing dimension-free preservers, we can assume that $T_n\neq\{[n]\}$ for infinitely many $n\geq 1$. Thus, as there is no loss of any generality, we assume that $T_{n}\neq \{[n]\}$ for all $n\geq 2$, in the results related to dimension-free classifications whenever needed.
\item[{3.}]
Given $I\subseteq \C$, the class $\Pn_n(I)$ is non-trivial only if $I\cap\R_{> 0}\neq \emptyset$, and nonempty only if $I\cap\R_{\geq 0}\neq \emptyset$.  We thus assume the corresponding necessities as and when needed.
\end{remark}

We now state the first result of this paper. See Table~\ref{Table1} for a summary of the classifications of $(g,f)$ according to the sequence $(T_n)_{n\geq 1}$.

\begin{utheorem}\label{T1}
Let $I=D(0,\rho)$ for $0<\rho\leq \infty$, and $g,f:I\to \C$ be functions. Let $T_n\subseteq 2^{[n]}$ for all $n\geq 1$ such that $(T_n)_{n\geq 1}$ is a nonempty sequence. Assuming $T_n\subseteq \{\{j\}:j\in [n]\}$ for all $n\geq 1$, the following are equivalent:
\begin{itemize}
\item[1.]
$(g,f)_{T_n}[A]\in\Pn_n$ for all $A\in\Pn_n(I)$, for all $n\geq 1$.
\item[2.]
$f(z)=\sum_{m,k\geq 0}c_{m,k}z^{m}\overline{z}^k$ for all $z\in I$, and $g(x)\geq f(x)$ for all $x\in I\cap \R_{\geq 0}$, where $c_{m,k}\geq 0$ for all $m,k\geq 0$.
\end{itemize}
\end{utheorem}

Theorem~\ref{T1} classifies the tuples $(g,f)$, such that $f$ is forbidden to operate on $1\times 1$ diagonal blocks (or diagonal entries) specified in $(T_n)_{n\geq 1}$, and $g$ operates on those forbidden blocks. As a consequence, $f$ now does not have to vanish at the origin, in contrast to Theorem~\ref{Guillot-Rajaratnam}. More precisely, we can now apply any absolutely monotonic function on the (specified) off-diagonal entries and any function on the (specified) diagonal entries by making sure that the latter function is large enough. Thus, the introduction of $g$ releases $f$ from a few restrictions, and makes the corresponding real analogue of Theorem~\ref{T1} (see Theorem~\ref{T1real}) into a twofold refinement of Theorem~\ref{Guillot-Rajaratnam}.

Next, we consider the possibilities when $f$ is forbidden from larger diagonal/principal blocks, i.e. $T_n \not\subseteq \{\{j\}:j\in [n]\}$ for some $n\geq 3$. Here we assume that $g$ is a Herz function.

\begin{utheorem}\label{T2}
Let $I= D(0,\rho)$ for $0<\rho\leq\infty$, and $g,f:I\to \C$ be functions. Let $T_n\subseteq 2^{[n]}$ for all $n\geq 1$ such that $(T_n)_{n\geq 1}$ is a nonempty sequence. Suppose there exists $N \geq 3$ and $U \in T_N$ with $2\leq |U| \leq N-1$. Assuming each $T_n$ is a partition of a subset of $[n]$, and $g(z):=\alpha z^{m}\overline{z}^k$ for $\alpha \geq 0$ and integers $m,k \geq 0$, the following are equivalent:
\begin{itemize}
\item[1.]
$(g,f)_{T_n}[A]\in\Pn_n$ for all $A\in\Pn_n(I)$, for all $n\geq 1$.
\item[2.]
Exactly one of the following holds:
\begin{itemize}
\item[a.] If $T_n$ is a partition of a proper subset of $[n]$ for some $n\geq 1$, then 
\begin{align*}
f(z)=cg(z) \mbox{ for all }z\in I, \mbox{ where }c\in [0,1].
\end{align*}
\item[b.] If $T_n$ is a partition of $[n]$ for all $n\geq 1$, and $K:=\max_{n\geq 1} |T_n| $, then
\begin{align*}
f(z)=cg(z) \mbox{ for all }z\in I,
\end{align*}
where $c\in [-1/(K-1),1]$ if $K\in \Z$, and $c\in [0,1]$ if $K=\infty$.
\end{itemize}
\end{itemize}
\end{utheorem}

In Theorem~\ref{T2}, we consider sequences $(T_n)_{n\geq 1}$ where each $T_n$ is a partition of some subset of $[n]$, and $f$ is forbidden from at least one $2\times 2$ diagonal/principal block, which means the Herz function $g$ in the tuple $(g,f)$ is acting (entrywise) on that block. As we show, this makes $f$ a scalar multiple of $g$, i.e. $f\equiv c \cdot g$. However, the domain for the scalar $c$ has a stronger dependence on the sequence $(T_n)_{n\geq 1}$, and this provides us with the first examples of dimension-free non-absolutely monotonic preservers (see Corollary~\ref{The-corollary}(2.c)). 

The next result relates the components in $(g,f)$ when $f$ is forbidden from at least two ``overlapping'' $2\times 2$ diagonal/principal blocks, i.e. the case when at least one $T_n$ is a not a partition of any subset of $[n]$ for some $n\geq 3$.

\begin{utheorem}\label{T3}
Let $I= D(0,\rho)$ for $0<\rho\leq\infty$, and $g,f:I\to \C$ be functions. Let $T_n\subseteq 2^{[n]}$ for all $n\geq 1$ such that $(T_n)_{n\geq 1}$ is a nonempty sequence. Assuming that $T_N$ is not a partition of any subset of $[N]$ for some $N\geq 3$, the following are equivalent:
\begin{itemize}
\item[1.]
$(g,f)_{T_n}[A]\in\Pn_n$ for all $A\in\Pn_n(I)$, for all $n\geq 1$.
\item[2.]
$g(z) = f(z)=\sum_{m,k\geq 0} c_{m,k}z^{m}\overline{z}^k$ for all $z\in I$, where $c_{m,k}\geq 0$ for all $m,k\geq 0$.
\end{itemize}
\end{utheorem}

In particular, Theorem~\ref{T3} shows that it is impossible to change one off-diagonal entry, independently from the rest of the entries, and claim to preserve positivity over the whole class of positive semidefinite matrices. (See Lemma~\ref{T3L1} and the subsequent remarks for finer details.)

The answer to Question~\ref{Question} follows from the results mentioned so far in this section (by substituting $g=\Id$). For completeness we include Theorem~\ref{HerzPlus} in the next corollary, that shows the classifications for all dimension-free cases possible among the sequences $(T_n)_{n\geq 1}$.

\begin{ucor}\label{The-corollary}
Let $I= D(0,\rho)$ for $0<\rho\leq\infty$, and $f:I\to\C$. Let $T_n\subseteq 2^{[n]}\setminus\{[n]\}$ for all $n\geq 1$. The following are equivalent:                                                                                    
\begin{itemize}
\item[1.] $f_{T_n}[A] \in \Pn_n$ for all $A\in\Pn_n(I)$, for all $n\geq 1$.
\item[2.] Exactly one of the following holds:

\begin{itemize}
\item[a.] If $T_n=\emptyset$ for all $n\geq 1$, then 
\begin{align*}
f(z)=\sum_{m,k\geq 0}c_{m,k}z^{m}\overline{z}^k \mbox{ for all } z\in I, \mbox{ where } c_{m,k}\geq 0 \mbox{ for all }m,k\geq 0. 
\end{align*}
\item[b.] If $T_n\neq\emptyset$ for some $n\geq 1$ and $T_n \subseteq \{\{j\}:j\in [n]\}$ for all $n\geq 1$, then 
\begin{align*}
f(z)=\sum_{m,k\geq 0}c_{m,k}z^{m}\overline{z}^k \mbox{ for all } z\in I \mbox{ and } f(x)\leq x \mbox{ for all } x\in I\cap\R_{\geq 0},
\end{align*}
$\mbox{where } c_{m,k}\geq 0$ for all $m,k\geq 0$.
\item[c.] If each $T_n$ is partition of a subset of $[n]$ and there exists $N\geq 3$ and $U\in T_N$ with $|U|\geq 2$, then
\begin{align*}
f(z)=cz  \mbox{ for all } z\in I, \mbox{ where }
\end{align*}
\begin{itemize}
\item[•]
$c\in [-1/(K-1),1]$ -- if $T_n$ is a partition of $[n]$ for all $n\geq 1$, where $K:=\max_{n\geq 1}|T_n|\in\Z$.

\item[•]
$c\in [0,1]$ -- if either $\max_{n\geq 1}|T_n|=\infty$ or $T_n$ is a partition of a proper subset of $[n]$ for some $n\geq 1$. 
\end{itemize}
\item[d.] If $T_n$ is not a partition of any subset of $[n]$ for some $n \geq 3$, then
\begin{align*}
f(z)=z \mbox{ for all } z\in I.
\end{align*}
\end{itemize}
\end{itemize}
\end{ucor}

This concludes the results for the complex disc domains. See Table~\ref{Table2} below that shows the conformity/contrast in the class of entrywise preservers $f$ characterized in Corollary~\ref{The-corollary}. It shows the necessary changes in the class of preservers $f$ as it progresses from operating on all the entries, to being forbidden from at least one $1\times 1$ diagonal block, to being forbidden from at least one $2\times 2$ diagonal/principal block. Since positive semidefinite matrices are closed under conjugation with permutation matrices, working with principal blocks is the same as working with diagonal blocks and vice versa.
\begin{table}[htbp]
\centering
\begin{tabular}{||c|c|c||}
\hline
\hline
&&\\
$f$ acts on  & $f$ is forbidden from  & $f$ is forbidden from  \\
all the entries & some  & some \\
& $1\times 1$ principal block & $2\times 2$ principal block\\
$\Updownarrow$ & $\Updownarrow$ & $\Updownarrow$ \\

$f$ is absolutely monotonic & $f$ is absolutely monotonic & $f$ is linear,\\
& which is  & vanishing at the origin,\\
& pointwise bounded-above & and sometimes\\
& by the & with \\
& identity function & negative coefficient\\
&&\\
\hline
\hline
\end{tabular}
\caption{Observe the contrast in the class of the positivity preservers when the size of the forbidden principal block is increased from $1\times 1$ to $2\times 2$. See Table~\ref{Table1} for the detailed version of this.} 
\label{Table2}
\end{table}

\subsection{Analogous proofs for real domains}\label{Introrealdomain}
Table~\ref{Table1} and results in the subsequent sections indicate that analogous proofs for $(g,f)$ hold over real domains $I=(-\rho,\rho)$, $[0,\rho)$ and $(0,\rho)$ for $0<\rho\leq\infty$. Corollary~\ref{The-corollary} thus follows immediately for these domains.

\begin{landscape}
\begin{table}[htbp]
\centering
\begin{tabular}{|c|c|c|c|c|c|}
\hline
 && $(g,f)$ & $(g,f)$ & $f$ & $f$ \\
\hline
 && $I= D(0,\rho),$ & $I_{\rho},$ & $I= D(0,\rho),$ & $I_{\rho},$ \\
 & $(T_n)_{n\geq 1}$ & where $0<\rho\leq \infty$ & where $0<\rho\leq \infty$ & where $0<\rho\leq \infty$ & where $0<\rho\leq \infty$ \\
\hline\hline
&&&&&\\
1. & $T_{n}=\emptyset$ & $f(z)=\sum_{m,k\geq 0}c_{m,k}z^{m}\overline{z}^k$ & $f(x)=\sum_{k\geq 0}c_{k}x^{k}$ & $f(z)=\sum_{m,k\geq 0}c_{m,k}z^{m}\overline{z}^k$ & $f(x)=\sum_{k\geq 0}c_{k}x^{k}$\\
&  for all $n\geq 1$ & where all $c_{m,k}\geq 0$ & where all $c_{k}\geq 0$ & where all $c_{m,k}\geq 0$ & where all $c_{k}\geq 0$ \\
\hline
&&&&&\\
2. & $T_{n}\subseteq\{\{j\}:j\in [n]\}$ & $f(z)=\sum_{m,k\geq 0}c_{m,k}z^{m}\overline{z}^k$ & $f(x)=\sum_{k\geq 0}c_{k}x^{k}$ & $f(z)=\sum_{m,k\geq 0}c_{m,k}z^{m}\overline{z}^k$ & $f(x)=\sum_{k\geq 0}c_{k}x^{k}$ \\
 & for all $n\geq 1$, and &  where all $c_{m,k}\geq 0$, &  where all $c_{k}\geq 0$, &  where all $c_{m,k}\geq 0$, &  where all $c_{k}\geq 0$, \\
 & $T_n\neq \emptyset$ for some $n\geq 2$ & $g(x)\geq f(x)$ over $I_{\geq 0}$  & $g(x)\geq f(x)$ over ${I_{\rho}}_{\geq 0}$  & $x\geq f(x)$ over $I_{\geq 0}$ & $x\geq f(x)$ over ${I_{\rho}}_{\geq 0}$ \\
\hline
&&&&&\\
3. &  $T_n=$ subpartition$([n])$& for $g(z)=\alpha z^{m}\overline{z}^k$ & for $g(z)=\alpha x^k$ & $f(z)=cz,$ where & $f(x)=cx,$ where \\
 & for all $n\geq 3$, and & where $\alpha\geq 0, m,k\in \Z_{\geq 0}:$ &  where $\alpha\geq 0, k\in \Z_{\geq 0}:$ && \\
 & $T_{n}\not\subseteq\{\{j\}:j\in [n]\}$ & &  &&\\
 & for some $n\geq 3$ & $f(z)=cg(z)$, where & $f(x)=cg(x)$, where &&\\
&&&&&\\
3.a & $\sqcup_{J\in T_n} J=[n]$& $c\in[-1/(K-1),1]$ & $c\in[-1/(K-1),1]$ & $c\in[-1/(K-1),1]$ & $c\in[-1/(K-1),1]$\\
    &  for all $n\geq 1$, and & &  & & \\ 
&$K:=\max_{n\geq 1}|T_n|\in\Z$&&&&\\
&&&&&\\
3.b & remaining sub-cases & $c\in [0,1]$& $c\in [0,1]$ & $c\in [0,1]$& $c\in [0,1]$\\
\hline
&&&&&\\
4. &  $T_n\neq$ subpartition$([n])$ &  $f(z)=g(z)=$ &  $f(x)=g(x)=$ & $f(z)=z$  &$f(x)=x$\\
 &  for some $n\geq 3 $ & $\sum_{m,k\geq 0}c_{m,k}z^{m}\overline{z}^k$, & $\sum_{k\geq 0}c_{k}x^{k}$,&(over any $I\subseteq\C$)& (over any $I\subseteq\R$)\\
& & where $c_{m,k}\geq 0$& where $c_{k}\geq 0$ &&\\
\hline
\end{tabular}
\caption{$(T_n)_{n\geq 1}$ against $(g,f)$ and $f$ for domains $I=D(0,\rho)$ and $I_{\rho}$, where $I_{\rho}$ is any of the real domains $(-\rho,\rho),(0,\rho)$ and $[0,\rho);$ define $I_{\geq 0}:=I\cap [0,\infty)$ and ${I_{\rho}}_{\geq 0}:={I_{\rho}}\cap [0,\infty).$ `subpartition$([n])$' refers to a partition of a subset of $[n]$. Here we study the dimension-free case, i.e. assuming each $T_n\neq\{[n]\}$; and for each $n$, the subsets in $T_n$ are pairwise incomparable.}
\label{Table1}
\end{table}
\end{landscape}

\section{Proof of Theorem~\ref{T1}}\label{ProofofT1}
In this section, we characterize the preserver tuples $(g,f)$ in full generality, given that each $T_n\subseteq \{\{j\}:j\in [n]\}$ in the given nonempty sequence $(T_n)_{n\geq 1}$, for complex disc domains. We need a few preliminary results. Our first lemma shows the relation between the components in $(g,f)$ when $(g,f)_{*}[-]$ preserves positive semidefiniteness (i.e. $T_n=\{\{j\}:j\in [n]\}$ for all $n\geq 1$), for any $I\subseteq \C$.

\begin{lemma}\label{T1L1}
Let $I\subseteq\C$ such that $I\cap\R_{\geq 0}\neq\emptyset$, and $g,f:I\to \C$. Then the following are equivalent:
\begin{itemize}
\item[1.] $(g,f)_{*}[A]\in\Pn_n$ for all $A\in\Pn_n(I)$, for all $n\geq 1$.
\item[2.] $f[A]\in\Pn_n$ for all $A\in\Pn_n(I)$ for all $n\geq 1$, and $g(x)\geq f(x)$ for all $x\in I\cap \R_{\geq 0}$.
\end{itemize}
\end{lemma}
\begin{proof}
($1$)$\implies$($2$): We adopt an argument in \cite{Guillot-Rajaratnam-Trans}. Let $n,m\geq 1$ be integers, and $A\in\Pn_n(I)$. Then,
\begin{align*}
(g,f)_{*}[\mathbf{1}_{m}\otimes A]&=f[\mathbf{1}_{m}\otimes A]+(g-f,\mathbf{0})_{*}[\mathbf{1}_{m} \otimes A] \\
&=\mathbf{1}_{m}\otimes f[A]+\Id_{m} \otimes (g-f,\mathbf{0})_{*}[A].
\end{align*}
Using Weyl's inequality for Hermitian matrices,
\begin{align*}
0\leq \lambda_{\min}((g,f)_{*}[\mathbf{1}_{m}\otimes A])&= \lambda_{\min}(\mathbf{1}_{m}\otimes f[A]+\Id_{m} \otimes (g-f,\mathbf{0})_{*}[A]) \\
&\leq \lambda_{\min}(\mathbf{1}_{m}\otimes f[A]) + \lambda_{\max}(\Id_{m} \otimes (g-f,\mathbf{0})_{*}[A]) \\
&\leq  m \lambda_{\min}(f[A])+\lambda_{\max}((g-f,\mathbf{0})_{*}[A]).
\end{align*} 
This gives us,
\begin{align*}
\lambda_{\min}(f[A])\geq -\frac{1}{m}\lambda_{\max}((g-f,\mathbf{0})_{*}[A]).
\end{align*}
Since $m$ can be arbitrarily large, 
\begin{align*}
\lambda_{\min}(f[A])\geq 0.
\end{align*}
This implies $f[A]\in\Pn_n$ for all $A\in\Pn_n(I)$ for all $n\geq 1$. Additionally, $g(x)\geq f(x)$ for all $x \in I\cap\R_{\geq 0}$ since $(g,f)_{*}[x\mathbf{1}_{2}]$ is positive semidefinite.\medskip

\noindent($2$)$\implies$($1$): Let $n \geq 1$ be an integer, and $A\in\Pn_n(I)$. Note that
\begin{align*}
(g,f)_{*}[A]=f[A]+(g-f,\mathbf{0})_{*}[A].
\end{align*}
Since, $g(x)\geq f(x)$ for all $x\in I\cap \R_{\geq 0}$, $(g-f,\mathbf{0})_{*}[A]\in\Pn_n$ and thus $(g,f)_{*}[A]\in\Pn_n$.
\end{proof}

Lemma~\ref{T1L1} shows that, over any given domain $I\subseteq \C$, $f$ must be the ``conventional'' entrywise positivity preserver in all dimensions for $(g,f)_{*}[-]$ to preserve positivity in all dimensions. However, the pointwise dominance of $g$ over $f$ follows just by assuming that $g$ acts on at least one diagonal entry in the matrices, i.e. $(T_n)_{n\geq 1}$ is nonempty:

\begin{lemma}\label{T1L2}
Let $I\subseteq\C$ such that $I\cap\R_{\geq 0}\neq\emptyset$, and $g,f:I\to \C$. Suppose $T_2=\{\{1\}\}$ or $T_2=\{\{1\},\{2\}\}$ and $(g,f)_{T_2}[A]\in \Pn_2$ for all $A \in \Pn_2(I)$. If $f(x)\geq 0$ for all $x\in I\cap \R_{\geq 0}$, then $g(x)\geq f(x)$ for all $x\in I\cap\R_{\geq 0}$.
\end{lemma}
\begin{proof}
Since $g$ is acting on a diagonal entry, it is nonnegative over $I\cap\R_{\geq 0}$. If $f(x)=0$ for $x\in I\cap\R_{\geq 0}$ then $g(x)\geq f(x)$, and if $f(x)\neq 0$ then by the positive semidefiniteness of $(g,f)_{T_2}[x\mathbf{1}_2]$, we have $g(x)\geq f(x)$.
\end{proof}

We now prove our first main result:

\begin{proof}[Proof of Theorem \ref{T1}]
We show that, if $(g,f)_{T_n}[A]\in\Pn_n$ for all $A\in\Pn_n(I),$ for all $n\geq 1,$ then -- either $f[A]\in\Pn_n$ for all $A\in\Pn_n(I)$ for all $n\geq 1$ or $(g,f)_{*}[A]\in\Pn_n$ for all $A\in\Pn_n(I)$ for all $n\geq 1$.

Note that $n-|T_n|$ is the number of diagonal entries of $(g,f)_{T_n}[A]$ which are equal to $f(a_{ii})$ for $i\in [n]$. We consider two cases of the sequence $(T_n)_{n\geq 1}$:

\begin{itemize}
\item[($\romannumeral 1\relax$)] 
$(n-|T_n|)_{n\geq 1}$ is unbounded: Without loss of generality, given a positive integer $n$ there exists a positive integer $N$ such that $N-|T_{N}| \geq n$. Let $A\in\Pn_n(I)$ and define the $N\times N$ matrix $A':=\begin{pmatrix}
A & \mathbf{0}\\
\mathbf{0} & \mathbf{0}
\end{pmatrix}$ by padding with zero matrices of sizes obvious from the context. Now we conjugate $A'$ with a suitable permutation matrix $P$ such that $f[A]$ is a principal submatrix of $(g,f)_{T_N}[P A' P^T]$, which must be positive semidefinite.

\item[($\romannumeral 2\relax$)] $(n-|T_n|)_{n\geq 1}$ is bounded: As in the previous case, we can embed $A$ and construct larger matrix $A^{\prime}$ such that the positive semidefinite principal submatrix is $(g,f)_{*}[A]$ for all $A\in \Pn_n(I)$ for all $n\geq 1$.
\end{itemize}

Now we use Lemma~\ref{T1L1} to conclude that $f[A]\in\Pn_n$ for all $A\in\Pn_n(I)$ for all $n\geq 1$ and then use Theorem~\ref{HerzPlus} to arrive at 
\begin{align*}
f(z)=\sum_{m,k\geq 0} c_{m,k} z^m \overline{z}^k  \mbox{ for all }z\in I, \mbox{ where } c_{m,k}\geq 0, \mbox{ for all }m,k\geq 0.
\end{align*}
Since $T_n$ is nonempty for some $n\geq 2$, and $f(x)\geq 0$ over $I\cap\R_{\geq 0}$, Lemma~\ref{T1L2} implies that
$
g(x)\geq f(x) \mbox{ for all }x\in I\cap\R_{\geq 0}.
$ Conversely, for $n\geq 1$ and $A\in\Pn_n(I),$
\begin{align*}
(g,f)_{T_n}[A]=f[A]+(g-f,\mathbf{0})_{T_n}[A].
\end{align*}
Using Theorem~\ref{HerzPlus} and that $g(x)\geq f(x)$ over $I\cap\R_{\geq 0}$, we conclude that $(g,f)_{T_n}[A] \in\Pn_n \mbox{ for all }A \in \Pn_n(I), \mbox{ for all } n\geq 1.$
\end{proof}

We relate Corollary~\ref{The-corollary}(2.a) and (2.b), i.e. $g=\Id$ in Theorem~\ref{T1}. We now have a proof that the preservers in (2.b) are the preservers in (2.a), which are absolutely monotonic for real domains (see Theorem~\ref{T1real}), that are moreover pointwise bounded-above by the function $g\equiv \Id$ over nonnegative reals in the domain. One of the notable consequences of this is that the preservers are necessarily linear (vanishing at the origin, with nonnegative slope) only if the domain $I$ includes all the positive real numbers, i.e. $I=\C$ for the complex disc domain case. This inclusion of all positive real numbers into the domain $I$ is not required for the linearity of the preservers in the coming sections, where $f$ is forbidden from at least one $2 \times 2$ diagonal/principal block, unlike the cases resolved in this section.

\subsection{Analogous proofs for real domains}\label{ProofofT1Subsec}
The proof of Theorem~\ref{T1} uses the fact that, $\Pn_n(I)$ can be embedded into $\Pn_N(I)$ for $n<N$ by padding the matrices in $\Pn_n(I)$ with zeros, for domains $I$ containing the origin. Thus the proofs for complex disc domains also go through verbatim for $I=(-\rho,\rho)$ or $[0,\rho)$. However, for $I=(0,\rho)$ padding with zeros is not possible, and thus we mention the next lemma for a workaround.

\begin{lemma}[Albert \cite{Albert}]
\label{Albert}
Let $I=(0,\rho)$ for $0<\rho\leq \infty$, and $n\geq 1$ be an integer. If $A=(a_{ij})\in\Pn_n(I)$ then for small $\epsilon > 0$, 
\begin{align*}
\begin{pmatrix}
A & \epsilon A\mathbf{1}_{n}\\
\epsilon (A\mathbf{1}_{n})^T & \epsilon \sum_{ij} a_{ij}
\end{pmatrix}\in\Pn_{n+1}(I).
\end{align*}
\end{lemma}

To complete the argument, we invoke either Theorem~\ref{SchoenbergPlus} or Theorem~\ref{VasudevaPlus}. The result can be stated formally as:

\begin{theorem}\label{T1real}
Let $I=(-\rho,\rho)$, $[0,\rho)$ or $(0,\rho)$ for $0<\rho\leq \infty$, and $g,f:I\to \R$ be functions. Let $T_n\subseteq 2^{[n]}$ for all $n\geq 1$ such that $(T_n)_{n\geq 1}$ is a nonempty sequence. Assuming $T_n\subseteq \{\{j\}:j\in [n]\}$ for all $n\geq 1$, the following are equivalent:
\begin{itemize}
\item[1.]
$(g,f)_{T_n}[A]\in\Pn_n$ for all $A\in\Pn_n(I)$, for all $n\geq 1$.
\item[2.]
$f(x)=\sum_{k\geq 0}c_{k}x^{k}$ for all $x\in I$, and $g(x)\geq f(x)$ for all $x\in I\cap \R_{\geq 0}$, where $c_{k}\geq 0$ for all $k\geq 0$.
\end{itemize}
\end{theorem}

\section{Proof of Theorem \ref{T2}}\label{ProofofT2}
This section classifies the entrywise positivity preservers $(g,f)$ for nonempty sequences $(T_n)_{n\geq 1}$ in which, unlike Section~\ref{ProofofT1}, $T_n \not \subseteq \{\{j\}:j\in [n]\}$ for some $n\geq 3$. However, we assume that each $T_n$ is a partition of some subset of $[n]$ for all $n\geq 1$, and $g$ is a Herz function. 

As the class of positive semidefinite matrices is closed under conjugation by permutation matrices and under taking principal submatrices, we gather necessary conditions on $f$ in the next two lemmas, by working with $(g,f)_{T_3}[-]$ for $T_3=\{\{1,2\}\}$ first, and then for $T_3=\{\{1,2\},\{3\}\}$.

\begin{lemma}\label{T2L1}
Let $I= D(0,\rho)$ for $0<\rho\leq\infty$, and $g,f:I\to\C$. Suppose $T_3=\{\{1,2\}\}$. Assuming $g(z):=\alpha z^{m}\overline{z}^k$ for $\alpha \geq 0$ and $m,k\in \Z_{\geq 0}$, the following are equivalent:
\begin{itemize}
\item[1.] $(g,f)_{T_3}[A]\in \Pn_3$ for all $A\in\Pn_3(I)$.
\item[2.] $(g,f)_{T_3}[A]\in \Pn_3$ for all rank-one $A \in \Pn_3(I)$.
\item[3.] $f(z)=cg(z)$ for all $z\in I$, where $ c \in [0,1]$.
\end{itemize}
\end{lemma}
\begin{proof} ($1$)$\implies$($2$) is obvious. We will prove ($2$)$\implies$($3$) and ($3$)$\implies$($1$).\medskip

\noindent($2$)$\implies$($3$): 
If $f\equiv 0$ or $g\equiv 0$, then there is nothing to prove. So suppose $f\not\equiv 0$ and $g\not\equiv 0$ (i.e., $\alpha \neq 0$), and note that as the operation $(g,f)_{T_3}[-]$ preserves positive semidefiniteness, we must have $f(\overline{z})=\overline{f(z)}$ for all $z\in I$, and $f(x)\geq 0$ for all $x\in I\cap \R_{\geq 0}$. We now relate $f(z)$ and $f(w)$ for $z,w\in I$ with $|z|\leq |w|$. Suppose $f(w)\neq 0$ for some $w\in I$. Then $f(t)>0$ for all $t\in [|w|,\rho)$, because 
\begin{align}\label{Mat1T2L1}
\begin{pmatrix}
g({|w|}^2/t) & g(w|w|/t) & f(w) \\
g(\overline{w} |w|/t) & g({|w|}^2/t) & f(|w|) \\
f(\overline{w}) & f(|w|) & f(t) \\
\end{pmatrix} \in\Pn_3 \mbox{ for all } t\in [|w|,\rho).
\end{align}
We thus pick $w\in I\setminus \{0\}$ such that $f(|w|) > 0$. Now, fix $z\in I$ such that $0 \leq |z| \leq |w| < \rho$, and define the rank-one matrix

\begin{align}\label{Mat2T2L1}
A_{w}(z):=\frac{1}{|w|}\begin{pmatrix}
z \\
w \\
w \\
\end{pmatrix}
\begin{pmatrix}
\overline{z} & \overline{w} & \overline{w} \\
\end{pmatrix}=
\begin{pmatrix}
{|z|}^2/{|w|} & z_1 & z_1\\
\overline{z_1} & |w| & |w| \\
\overline{z_1} & |w| & |w| \\
\end{pmatrix},
\mbox{ where }z_1:=z\dfrac{\overline{w}}{|w|}.
\end{align}
Since $A_{w}(z) \in \Pn_3(I)$, 
\begin{align*}
(g,f)_{T_3}[A_{w}(z)]=
\begin{pmatrix}
g({|z|}^2/{|w|}) & g(z_1) & f(z_1)\\
\overline{g(z_1)} & g(|w|) & f(|w|) \\
\overline{f(z_1)} & f(|w|) & f(|w|) \\
\end{pmatrix}\in\Pn_3.
\end{align*}
Using the fact that $g(z)=\alpha z^{m} \overline{z}^k$, the Schur complement of $((g,f)_{T_3}[A_{w}(z)])_{33}$ equals
\begin{align*}
\begin{pmatrix}
g({|z_1|}^2)/g({|w|})-{|f(z_1)|}^2/{f(|w|)} & g(z_1)-f(z_1) \\
\overline{g({z_1})}-\overline{f(z_1)} & g(|w|)-f(|w|) \\
\end{pmatrix},
\end{align*} 
the determinant of which is
\begin{align*}
-\Bigg{|}g(z_1)\sqrt{\dfrac{f(|w|)}{g(|w|)}} - f(z_1)\sqrt{\dfrac{g(|w|)}{f(|w|)}}\Bigg{|}^2.
\end{align*}
This determinant is nonnegative, so
$f(z_1)=\dfrac{f(|w|)}{g(|w|)} g(z_1)$.
Since $z \mapsto z_1=z\dfrac{\overline{w}}{|w|}$ is a bijection and it preserves the norm,

\begin{align}\label{OnlyOneSquare-Complex-eqn}
f(z)=\dfrac{f(|w|)}{g(|w|)} g(z) \mbox{ where } 0 \leq |z| \leq |w|<\rho, \mbox{ whenever } f(|w|)> 0 \mbox{ for } |w| > 0.
\end{align} 

Since $0 < \dfrac{f(|w|)}{g(|w|)}\leq 1$ whenever $f(|w|)>0$ for $|w|>0$, combining \eqref{Mat1T2L1} and \eqref{OnlyOneSquare-Complex-eqn}, we have 
\begin{align*}
f(z)=cg(z) \text{ for all } z\in I,\text{ where } c\in [0,1].
\end{align*}
($3$)$\implies$($1$):
For $A\in\Pn_3(I)$, using the Schur Product Theorem~\ref{Schur}, we have,
\begin{align*}
(g,f)_{T_3}[A] = f[A]+ (g-f,\mathbf{0})_{T_3}[A]=cg[A]+(1-c)(g,\mathbf{0})_{T_3}[A]\in \Pn_3,
\end{align*}
for $c\in[0,1]$. This proves the lemma.
\end{proof}

\begin{lemma}\label{T2L2}
Let $I= D(0,\rho)$ for $0<\rho\leq\infty$, and $g,f:I\to\C$. Suppose $T_3=\{\{1,2\},\{3\}\}$. Assuming $g(z):=\alpha z^{m}\overline{z}^k$ for $\alpha\geq 0$ and $m,k\in\Z_{\geq 0}$, the following are equivalent:
\begin{itemize}
\item[1.] $(g,f)_{T_3}[A]\in\Pn_3$ for all $A\in\Pn_3(I)$.
\item[2.] $(g,f)_{T_3}[A]\in\Pn_3$ for all rank-one $A\in\Pn_3(I)$.
\item[3.] $f(z)=cg(z)$ for all $z\in I$, where $c\in[-1,1]$.
\end{itemize}
\end{lemma}
\begin{proof}
($1$)$\implies$($2$) is obvious; we will show ($2$)$\implies$($3$) and ($3$)$\implies$($1$).\medskip

\noindent($2$)$\implies$($3$): If $f\equiv 0$ or $g\equiv 0$, then there is nothing to prove. So suppose $f\not\equiv 0$ and $g\not\equiv 0$ (i.e., $\alpha\neq 0$), and note that as $(g,f)_{T_3}[-]$ preserves positive semidefiniteness, we must have $f(\overline{z})=\overline{f(z)}$ for all $z\in I$. Let $z,w\in I$ be such that $|z|\leq |w|$ and $|w| > 0$. For $A_{w}(z)$ defined in \eqref{Mat2T2L1}, 
\begin{align*}
(g,f)_{T_3}[A_{w}(z)]=
\begin{pmatrix}
g({|z|}^2/{|w|}) & g(z_1) & f(z_1)\\
\overline{g(z_1)} & g(|w|) & f(|w|) \\
\overline{f(z_1)} & f(|w|) & g(|w|) \\
\end{pmatrix}\in\Pn_3.
\end{align*}
The Schur complement of $((g,f)_{T_3}[A_{w}(z)])_{33}$ is
\begin{align*}
B:=\frac{1}{g(|w|)}\begin{pmatrix}
|g({z_1})|^2-{|f(z_1)|}^2 & g(|w|) g(z_1)-f(|w|) f(z_1) \\
g(|w|) \overline{g(z_1)}-f(|w|)\overline{f(z_1)} & {g(|w|)}^2-{f(|w|)}^2 \\
\end{pmatrix}\in\Pn_2 
\end{align*} 
and the determinant
\begin{align*}
\det(B)= -\frac{1}{{g(|w|)}^2} \big{|}f(|w|)g(z_1)-g(|w|)f(z_1)\big{|}^2.
\end{align*}
This determinant is nonnegative, and so we have
$
f(z_1)=\dfrac{f(|w|)}{g(|w|)} g(z_1);
$
 since $z \mapsto z_1=\dfrac{\overline{w}}{|w|} z$ is a bijection which preserves the norm, we obtain
\begin{align*}
f(z)=\dfrac{f(|w|)}{g(|w|)}g(z), \qquad \mbox{whenever }  |z| \leq |w|< \rho \mbox{ and }|w|>0.
\end{align*}
Let $c:=\dfrac{f(|w|)}{g(|w|)}$; then $c\in[-1,1]$, as $B\in\Pn_2$. Hence, 
\begin{align*}
f(z)=cg(z) \mbox{ for all } z\in I, \mbox{ where } c\in[-1,1].
\end{align*}
($3$)$\implies$($1$):
For $A\in\Pn_3(I)$,
\begin{align*}
(g,f)_{T_3}[A] =  f[A]+(g-f,\mathbf{0})_{T_3}[A]=cg[A]+(1-c)(g,\mathbf{0})_{T_3}[A].
\end{align*}
Clearly $(g,f)_{T_3}[A]$ is positive semidefinite for $c\in [0,1]$. Since the principal minors of $(g,f)_{T_3}[A]$ are functions of $c^2$, it is positive semidefinite for $c\in [-1,0)$ too, which completes the proof.
\end{proof}

Combining Lemma~\ref{T2L1} and Lemma~\ref{T2L2}, we conclude that as soon as $f$ is forbidden from some $2\times 2$ diagonal/principal block, i.e. there is a non-singleton set in $T_n$ for some $n\geq 3$, $f$ is a scalar multiple of the given Herz function $g$ (for some $c\in [-1,1]$). However, the discussion on the scalar needs more attention. For instance in Lemma~\ref{T2L1}, $f$ must preserve $\R_{\geq 0}$ since it is allowed to operate on one of the diagonal entries, and so the scalar $c$ has to be nonnegative. This is not the case in Lemma~\ref{T2L2}, where the scalar $c$ could also be negative. 

The next two results provide the exact possible range of the constant $c$ when each $T_n$ is a partition of $[n]$ and when $g(z)=z$. 

\begin{lemma}\label{LonlyDiag}
Let $I\subseteq\C$ such that $I\cap\R_{> 0}\neq \emptyset$, and $f(z) = cz$  for all $z\in I$ for some $c \in \R$, and suppose $n\geq 2$ is an integer. The following are equivalent:
\begin{itemize}
\item[1.] $f_{*}[A]\in \Pn_n$ for all $A\in\Pn_n(I)$.
\item[2.] $f_{*}[A]\in \Pn_n$ for all rank-one $A\in \Pn_n(I)$.
\item[3.] $c \in [-1/(n-1),1]$.
\end{itemize}
\end{lemma}
\begin{proof}
($1$)$\implies$($2$) is obvious. We start with\\
($2$)$\implies$($3$): For $x\in I\cap\R_{\geq 0}$ the matrix
$
f_{*}[x {\bf 1}_{n}] = c x {\bf 1}_{n} + (1-c) x \Id_n \in \Pn_n.
$
But this matrix has eigenvalues $(1-c)x$ and $(1 + (n-1)c)x$. As these are nonnegative, we obtain $c \in [-1/(n-1),1]$.\\
($3$)$\implies$($1$): Given a square matrix $A$, note that ${\bf 0}_{*}[A]$ is the diagonal matrix having entries $a_{ii},\ i \geq 1$, on the diagonals. Now if $c \in [0,1]$ then for any matrix $A \in \Pn_n(I)$,
\begin{align*}
f_{*}[A] = c A + (1-c) {\bf 0}_{*}[A] \in \Pn_n.
\end{align*}
If instead $c \in [-1/(n-1),0)$, then we write
\begin{align*}
f_{*}[A] = (1 + (n-1)c) {\bf 0}_{*}[A] +  |c| (n {\bf 0}_{*}[A] - A).
\end{align*}
Clearly the first term is positive semidefinite, so it suffices to show that
\begin{align*}
n {\bf 0}_{*}[A] - A \in \Pn_n, \qquad \forall A \in \Pn_n(I).
\end{align*}
For this we use a `correlation trick': since $n$ is fixed, we may perturb $A$ by $\epsilon \Id_n$ for $\epsilon > 0$, thereby assuming $A, {\bf 0}_{*}[A]$ are both positive definite. (The result for such $A$ implies the result for all $A \in \Pn_n(I)$ by sending $\epsilon \to 0^+$.) Now pre- and post- multiplying by the diagonal matrix ${\bf 0}_{*}[A]^{-1/2}$, it suffices to show that $n \Id_n - C \in \Pn_n$, for all correlation matrices $C_{n \times n}$ (i.e., positive matrices with ones on the diagonal). But this is immediate; we provide two proofs. First,
\begin{align*}
\lambda_{\min}(n \Id - C) = n - \lambda_{\max}(C) \geq n - {\rm tr}(C) =
0.
\end{align*}
Alternately, $n \Id_n - C$ has real eigenvalues, and is diagonally
dominant, so we are done by Gershgorin's circle theorem.
\end{proof}

The following result, due to Khare (personal communication), is a generalization of Lemma~\ref{LonlyDiag} to arbitrary partitions $T_n$ of $[n]$.

\begin{theorem}\label{GenLonlyDiag}
Let $I\subseteq \C$ such that $I\cap\R_{> 0}\neq\emptyset$. Suppose $n \geq 2$ is an integer, and $T_n \subseteq 2^{[n]}$ is a partition of $[n]$ into $k \geq 2$ subsets.
Assuming that $f(z) = cz$ for all $z\in I$ for some $c \in \R$, the following are equivalent:
\begin{itemize}
\item[1.]$f_{T_n}[A]\in\Pn_n$ for all $A\in\Pn_n(I)$.
\item[2.]$f_{T_n}[A]\in\Pn_n$ for all rank-one $A\in \Pn_n(I)$.
\item[3.]$c \in [-1/(k-1),1]$.
\end{itemize}
\end{theorem}

\begin{proof}$(1)\implies(2)$ is trivial, and ($2$)$\implies$($3$) follows by applying Lemma~\ref{LonlyDiag} to a principal submatrix of $f_{T_n}[A]$. We prove,
($3$)$\implies$($1$): if $c \in [0,1]$ then $f_{T_n}[A]$ is positive semidefinite for all $A\in\Pn_n(A)$:
\begin{align*}
f_{T_n}[A] = c A + (1-c) \mathbf{0}_{T_n}(A)   \in \Pn_n.
\end{align*}
Henceforth, we thus suppose $c \in [-1/(k-1),0)$, and let $T_n = \{ J_1, \dots, J_k \}$ with $J_1 \sqcup \cdots \sqcup J_k = [n]$. Since positive semidefinite matrices are closed under conjugation with permutation matrices, it is enough to prove (1) for contiguous $J_i$, i.e., $J_1=\{1,2,\dots,n_{j_1}\}$, $J_2=\{n_{j_1}+1,n_{j_1}+2,\dots,n_{j_2}\}$ and so on. We will write matrices in $\Pn_n(I)$ in block-form, corresponding to the $J_j$, say $A =(A_{ij})_{i,j = 1}^k$. Notice that
\begin{align*}
f_{T_n}[A] = (1 + |c|) \mathbf{0}_{T_n}(A) - |c| A = (B_{ij})_{i,j=1}^k,  \text{ where }  
B_{ij} = 
\begin{cases}
A_{ij}, & \text{if } i=j,\\
c A_{ij}, & \text{if } i \neq j.
\end{cases}
\end{align*}
We may use a perturbation by adding $\epsilon \Id_n$ to $A$ and letting
$\epsilon \to 0^+$. Since $f_{T_n}[A + \epsilon \Id_n] = f_{T_n}[A] + \epsilon \Id_n$, it therefore suffices to show the result under the assumption that each $A_{ii}$, and hence $A$, is positive definite. We assume this henceforth.

We now show the result by induction on $k \geq 2$, with no restrictions on $n \geq k$ except that $T_n$ contains $k$ subsets. If $k=2$, then we are to show
\begin{align*}
\begin{pmatrix}
A & B \\
B^* & C
\end{pmatrix} \in \Pn_n(I)
\qquad \implies \qquad
\begin{pmatrix}
A & cB \\
cB^* & C
\end{pmatrix} \in \Pn_n
\end{align*}
for $c \in [-1,1]$, assuming that $A,C$ are invertible. But this is immediate by using Schur complements: the hypotheses imply
\begin{align*}
C - B^* A^{-1} B \in \Pn_n(I) \qquad \implies \qquad C - c^2 B^* A^{-1} B \in \Pn_n,
\end{align*}
which implies the result for $k=2$.

For the induction step, suppose we know the result for $k \geq 2$, and $T_n$ consists of $k+1$ parts, say $J_1, \dots, J_{k+1}$. Define
\begin{align*}
m := \sum_{j=1}^k |J_j|, \qquad T_m := \{ J_1, \dots, J_k \}
\end{align*}
without loss of generality; then we know by the induction hypothesis that
\begin{align}\label{Eindhyp}
c' \in [-1/(k-1),0), \ \ h(z) := c'z \implies
h_{T_m}[A'] \in \Pn_m\ \ \mbox{ for all }A' \in \Pn_m(I).
\end{align}

Now suppose $A = (A_{ij})_{i,j=1}^{k+1} \in \Pn_n(I)$. Write
\begin{align*}
A = \begin{pmatrix} A' & B \\ B^* & A_{k+1,k+1} \end{pmatrix}, \qquad
\text{where} \quad A' := A_{[m] \times [m]} = (A_{ij})_{i,j=1}^k \in
\Pn_m(I).
\end{align*}
Hence
\begin{align*}
f_{T_n}[A] = \begin{pmatrix} f_{T_m}[A'] & cB \\ cB^* & A_{k+1,k+1}
\end{pmatrix}.
\end{align*}
Now using Schur complements, we have
\begin{align*}
A \in\Pn_n \implies  A'-B A_{k+1,k+1}^{-1} B^* \in\Pn_m 
&\implies c^2 A' - c^2 B A_{k+1,k+1}^{-1} B^* \in \Pn_m \\
&\implies \begin{pmatrix} c^2 A' & cB \\ c B^* & A_{k+1,k+1}
\end{pmatrix} \in\Pn_n,
\end{align*}
where we use that $|c| < 1$.

Thus the proof is complete if we can show that $f_{T_m}[A'] - c^2 A' \in \Pn_m$, for $c \in [-1/k,0)$. Note here that $|c| < 1$ since $k \geq 2$.

Now one computes:
\begin{align*}
\frac{1}{1-c^2} (f_{T_m}[A'] - c^2 A') = (B_{ij})_{i,j=1}^k, \qquad
\text{where} \qquad B_{ij} = \begin{cases}
A_{ij}, & \text{if } i=j,\\
\frac{c}{1+c} A_{ij}, & \text{if } i \neq j.
\end{cases}
\end{align*}
It is easily verified that
\begin{align*}
c \in [-1/k, 0) \quad \Longleftrightarrow \quad
c' \in [-1/(k-1),0), \quad \text{where } c' := \frac{c}{1+c}.
\end{align*}
Thus if we define $h(z) := c'z$, then from above, we have
\begin{align*}
f_{T_m}[A'] - c^2 A' = (1-c^2) h_{T_m}[A'],
\end{align*}
and this is positive semidefinite by \eqref{Eindhyp}. This shows the induction step, and concludes the proof.
\end{proof}

\begin{remark}
The implication $(3)\implies (1)$ in Theorem~\ref{GenLonlyDiag} has an alternate simpler proof using the Schur Product Theorem~\ref{Schur}. Note that $f_{T_n}[A]=A \circ f_{T_n}[\mathbf{1}]$, thus it is enough to show that $f_{T_n}[\mathbf{1}]$ is positive semidefinite. The principal submatrices of $f_{T_n}[\mathbf{1}]$ which are of the form $f_{*}[\mathbf{1}]$ are positive semidefinite by Lemma~\ref{LonlyDiag}, and the remaining principal submatrices have two identical rows, so their determinant vanishes.
\end{remark}

With Theorem~\ref{GenLonlyDiag}, we can now see the dependence of the scalar $c$ on the sequence $(T_n)_{n\geq 1}$, which consists only of partitions of $[n]$ for all $n\geq 1$. The scalar $c$ depends on $\max_{n\geq 1}|T_n|$, a ``global'' property of the sequence $(T_n)_{n\geq 1}$. It is also worth noting that when $\max_{n\geq 1}|T_n|=\infty$, the scalar $c\in [0,1]$; this phenomena shows that $c\geq 0$ (i.e., $f$ preserves $\R_{\geq 0}$) even when $f$ does not have to act on any of the diagonal entries.

To complete this section, we now prove Theorem~\ref{T2}.

\begin{proof}[Proof of Theorem~\ref{T2}]
($1$)$\implies$($2$.a): We first use Lemma~\ref{T2L1} or Lemma~\ref{T2L2} to obtain that $f(z)=cg(z) \mbox{ for all } z\in I, \mbox{ for some } c\in[-1,1].$ However, here $f$ preserves $\R_{\geq 0}$ as it operates on some diagonal entry, implying
$
f(z)=cg(z) \mbox{ for all } z\in I, \mbox{ for some } c\in[0,1].
$\medskip

\noindent($2$.a)$\implies$($1$): For $c\in [0,1]$,
\begin{align*}
(g,f)_{T_n}[A]=f[A]+(g-f,\mathbf{0})_{T_n}[A]=cg[A]+(1-c)(g,\mathbf{0})_{T_n}[A]\in\Pn_n
\end{align*}
for all $A\in\Pn_n(I)$ for all $n\geq 1$.\medskip

\noindent($1$)$\implies$($2$.b): First we invoke Lemma~\ref{T2L2} to conclude that
$
f(z)=cg(z) \mbox{ for all } z\in I,\mbox{ for some}$ $ c\in[-1,1]
$. Now, as $K:=\max_{n\geq 1}|T_n|\in \Z$ we let $|T_N|=K$ for some $N\geq 2$; note that $K\geq 2$. As $(g,f)_{*}[x\mathbf{1}_K]$ is a principal submatrix of $(g,f)_{T_N}[x \mathbf{1}_N]$, hence is positive semidefinite for $ x\in I \cap\R_{\geq 0}$. Invoke Theorem~\ref{GenLonlyDiag} to conclude that $c \in [-1/(K-1),1]$. Similarly, if $K=\infty$ then $c\in [0,1] $.\medskip

\noindent($2$.b)$\implies$($1$):
First suppose $K\geq 2$ is finite, and let $c\in[-1/(K-1),1]$. For $A\in\Pn_n(I)$ for $n\geq 1$, we write 
\begin{align*}
(g,f)_{T_n}[A]=g[A]\circ (\mathbf{1},c\mathbf{1})_{T_n}[A].
\end{align*}
We use Theorem~\ref{GenLonlyDiag} and Theorem~\ref{Schur} to conclude that $(g,f)_{T_n}[A]\in\Pn_n$. When $K=\infty$, we use an argument similar to ($2$.a)$\implies$($1$). This concludes the proof.
\end{proof}

To summarise this section: $f$ is forbidden from acting on at least one $2\times 2$ diagonal/principal block (and a given Herz function $g$ acts on such blocks), and no two such forbidden blocks share a common entry of the matrices. This restriction itself forces $f$ to be a scalar multiple of $g$, i.e. $f \equiv c\cdot g$. With this ``linearity'' at hand, we have shown that the scalar $c \in [0,1]$ if $f$ acts on some diagonal entry (as $f$ must preserve $\R_{\geq 0}$). In other cases, when $f$ is not acting on any diagonal entry, i.e. each $T_n$ is a partition of $[n]$, we saw that the scalar $c\in [-1/(K-1),1]$ where $K:=\max_{n\geq 1}|T_n|\in \Z$. This adds non-absolutely monotonic preservers $f$ (for $g=\Id$, see Corollary~\ref{The-corollary}(2.b)). Finally, if the size of the partition grows without bound, i.e. $\max_{n\geq 1}|T_n|=\infty$, then $c\in [0,1]$.

We will see that the scalar $c$ mentioned in the previous paragraph is necessarily equal to $1$ in the cases resolved in Section~\ref{ProofofT3}. More precisely, we will see that as soon as two forbidden $2\times 2$ blocks share a common (diagonal) entry, $g\equiv f$ over any domain $I\subseteq \C$.

\subsection{A fixed-dimensional result and analogous proofs for real domains}\label{ProofofT2Subsec}
The following is a fixed-dimensional result that follows from the results discussed so far in this section.

\begin{theorem}\label{fixed-dim}
Let $I=D(0,\rho)$ for $0<\rho\leq\infty$, and $g,f:I\to\C$. Suppose $n \geq 2$ be an integer, and $T_n\subseteq 2^{[n]}$ is a partition of $[n]$ such that the cardinality $|T_n|\in\{2,\dots,n-1\}$. Assuming $g(z)=\alpha z^{m}\overline{z}^k$ for $\alpha\geq 0$ and $m,k\in\Z_{\geq 0}$, the following are equivalent:

\begin{itemize}
\item[1.] $(g,f)_{T_n}[A] \in \Pn_n$ for all $A\in\Pn_n([0,\infty))$.
\item[2.] $f(z)=cg(z)$ for all $z\in I$, where $ c \in [-1/(|T_n|-1),1]$.
\end{itemize}
\end{theorem}
\begin{proof}
Suppose $(1)$ holds. By using Lemma~\ref{T2L2} we can deduce that $f(z)=cg(z)$ for all $z\in I$, where $c\in [-1,1]$. As a principal submatrix of $(g,f)_{T_n}[x\mathbf{1}]$, for $x\in I\cap \R_{>0}$, equals $(g,f)_{*}[x {\bf 1}] = \alpha x^{m+k}(c {\bf 1_{|T_n|}} + (1-c)  \Id_{|T_n|})$ and is positive semidefinite, using the arguments in Lemma~\ref{LonlyDiag}, we have $(2)$. And $(2)\implies (1)$ follows from Theorem~\ref{GenLonlyDiag} and the Schur Product Theorem~\ref{Schur}.
\end{proof}
\begin{remark}
\item[1.] By following the proofs in this section, results analogous to Theorem~\ref{T2} can be proved similarly for the domains $I=(-\rho,\rho)$, $[0,\rho)$ and $(0,\rho)$ for $0<\rho\leq\infty$. See the row numbered 3 in Table~\ref{Table1} for this classification.
\item[2.] Theorem~\ref{fixed-dim} can also be proved similarly for the real domains mentioned in the previous remark. Moreover, we note that the implication $(1)\implies (2)$ for $[0,\rho)$ and $(0,\rho)$ will follow for $g(x)=\alpha x^{\beta}$ for all $\alpha,\beta\geq 0$.
\end{remark}

\section{Proof of Theorem~\ref{T3}}\label{ProofofT3}
This section classifies the preservers $(g,f)$ for the remaining kind of sequences $(T_n)_{n\geq 1}$, that are not discussed above, i.e. when some $T_n$ is not a partition of a subset of $[n]$. We therefore assume below that some $T_n$ has overlapping blocks, i.e., there exist $U, V \in T_n$ with $U\neq V$ and $U \cap V \neq \emptyset$. As in the previous section, we first gather necessary conditions on $(g,f)$ by working in low dimension.

\begin{lemma}\label{T3L1}
Let $I\subseteq \C$ be closed under complex conjugation such that $I\cap \R_{>0}\neq \emptyset$, and $g,f:I\to\C$. Suppose $T_3=\{\{1,2\},\{2,3\}\}$. Assuming that $f[A]\in\Pn_3$ for all $A\in\Pn_3(I)$ the following are equivalent:
\begin{itemize}
\item[1.] $(g,f)_{T_3}[A]\in\Pn_3$ for all $A\in\Pn_3(I)$.
\item[2.] $f(z)=g(z)$ for all $z\in I$ such that $|z|\leq r$ for some $r\in I\cap\R_{>0}$.
\end{itemize}
Moreover, $(1)\implies(2)$ follows without the assumption that $f[A]\in\Pn_3$ for all $A\in\Pn_3(I)$.
\end{lemma}
\begin{proof} ($2$)$\implies$($1$) is obvious.\\
($1$)$\implies$($2$): Let $r\in I \cap \R_{>0}$, and $z\in I\cap \overline{D}(0,r)$, where $\overline{D}(0,r)$ denotes the closed disk of radius $r$ centered at the origin. Since $(g,f)_{T_3}[-]$ preserve positive semidefiniteness, $g$ and $f$ must preserve the conjugation operation, and we have the following,
\begin{align*}
B_{r}(z):=\begin{pmatrix}
r & z & z \\
\overline{z} & r & r \\
\overline{z} & r & r \\
\end{pmatrix}\in\Pn_3(I)
\implies
(g,f)_{T_3}[B_{r}(z)]=
\begin{pmatrix}
g(r) & g(z) & f(z) \\
\overline{g(z)} & g(r) & g(r) \\
\overline{f(z)} & g(r) & g(r) \\
\end{pmatrix}\in\Pn_3.
\end{align*}
Thus its determinant is nonnegative for all $z\in I\cap \overline{D}(0,r)$, and for all $r\in I\cap\R_{>0}$, and is given by,
\begin{align*}
\det((g,f)_{T_3}[B_{r}(z)])&=0-g(z)\big{(}g(r)\overline{g(z)}-g(r)\overline{f(z)}\big{)}+f(z)\big{(}g(r)\overline{g(z)}-g(r)\overline{f(z)}\big{)}\\  
&=-g(r)\big{(}f(z)-g(z)\big{)}\big{(}\overline{f(z)}-\overline{g(z)}\big{)} \\
&=-g(r)\big{|}f(z)-g(z)\big{|}^2
\end{align*}
If $g(r)=0$ then $f(z)=g(z)=0$, else $f(z)=g(z)$. Hence
\begin{align*}
f(z)=g(z), \qquad\mbox{for all } z\in I\cap \overline{D}(0,r), \qquad\mbox{for all }r\in I\cap\R_{>0}
\end{align*}
as desired. 

\end{proof}

Applying Lemma~\ref{T3L1} to principal submatrices of $N\times N$ matrices, we obtain that, as soon as some $T_N$, $N\geq 3$ is not a partition of any subset of $[N]$, the functions in the preserver-tuple $(g,f)$ have to be identical. In other words, it is impossible to change one off-diagonal entry independently of the rest of the entries and claim to preserve positivity over $\Pn_N(I)$. 

This equivalence is essentially true over any domain $I\subseteq\C$. To be precise: $g\equiv f$ over the subset of the domain $I$ that comprises of the elements that occur as an entry in some matrix in $\Pn_N(I)$ (which is the best that one can do!). For example, $g\equiv f$ over $I=D(0,\rho),(-\rho,\rho),(0,\rho)$ or $[0,\rho)$ for $0< \rho \leq \infty$, or over $I=V\cup \overline{V}\cup\R_{\geq 0}$ for any nonempty subset $V\subseteq \C$, where $\overline{V}:=\{\overline{z}:z\in V\}$.

We are now ready for the proof of our final main result.

\begin{proof}[Proof of Theorem~\ref{T3}]
By assumption, there exists $N\geq 3$ such that $T_N$ is not a partition of any subset of $[N]$, say $T_{N}=\{\{m,m_1,\dots\},\{m,m_2,\dots\},\dots\}$ for some distinct $m,m_1,m_2 \in [N]$. Since the operation $(g,f)_{T_N}[-]$ preserves positive semidefiniteness over $\Pn_N(I)$, as a sub-operator, $(g,f)_{T_3}[-]$ must preserve the positive semidefiniteness over $\Pn_3(I)$ for $T_3=\{\{1,2\},\{2,3\}\}$. Thus, $f(z)=g(z)$ for all $z\in I$ using Lemma~\ref{T3L1}, and so the operator $(g,f)_{T_n}[-]\equiv f[-]$, and thus by Theorem~\ref{HerzPlus} we conclude that
\begin{align*}
g(z)=f(z)=\sum_{m,k}c_{m,k}z^{m}\overline{z}^k \mbox{ for all }z\in I, \mbox{ where }c_{m,k}\geq 0 \mbox{ for all }m,k\geq 0.
\end{align*}
The converse follows from Theorem~\ref{HerzPlus}.
\end{proof}

\subsection{Analogous proofs for real domains}\label{ProofofT3Subsec}
Theorem~\ref{T3} can be similarly proved for the domains $I=(-\rho,\rho)$, $[0,\rho)$ and $(0,\rho)$ for $0<\rho\leq\infty$ by using Lemma~\ref{T3L1}. We invoke Theorem~\ref{SchoenbergPlus} when $I=(-\rho,\rho)$ and Theorem~\ref{VasudevaPlus} when $I=[0,\rho)$ or $(0,\rho)$, for $0<\rho\leq\infty$. More precisely: 
\begin{theorem}\label{T3real}
Let $I= (-\rho,\rho)$, $[0,\rho)$ or $(0,\rho)$ for $0<\rho\leq\infty$, and $g,f:I\to \R$ be functions. Let $T_n\subseteq 2^{[n]}$ for all $n\geq 1$ such that $(T_n)_{n\geq 1}$ is a nonempty sequence. Assuming $T_N$ is not a partition of any subset of $[N]$ for some $N\geq 3$, the following are equivalent:
\begin{itemize}
\item[1.]
$(g,f)_{T_n}[A]\in\Pn_n$ for all $A\in\Pn_n(I)$, for all $n\geq 1$.
\item[2.]
$g(x) = f(x)=\sum_{k\geq 0} c_{k}x^{k}$ for all $x\in I$, where $c_{k}\geq 0$ for all $k\geq 1$.
\end{itemize}
\end{theorem}

This paper, apart from providing the first examples of non-absolutely monotonic dimension-free entrywise preservers, introduces the notion of tuples $(g,f)$ acting entrywise on matrices and preserving positivity. This idea of operating by two functions in the entrywise fashion gives an opportunity to understand the role of diagonal/principal blocks in positive semidefiniteness, among other things. For instance, in Section~\ref{ProofofT2}, it is shown that taking $g$ to be a Herz function forces $f$ to be a scalar multiple of it, where the scalar depends on the diagonal/principal blocks. This raises the further question as to what happens to $f$ when $g$ is not necessarily a Herz function. With this remark, we conclude this paper.

\subsection*{Acknowledgements}
I am grateful to my Ph.D. supervisor, Apoorva Khare, for thought-provoking and encouraging discussions, for meticulously going through the preliminary drafts of the paper, and for valuable feedback. Part of this work was carried out at the University of Regina (Canada) and I am grateful to the Canadian Queen Elizabeth II Diamond Jubilee Scholarship (QES) for supporting my visit. I also thank Shaun Fallat for stimulating discussions. Finally, I sincerely appreciate the support of the anonymous referee who carefully read the paper and provided many suggestions that helped me improve the manuscript.

\bibliographystyle{plain}
\bibliography{pds}

\end{document}